\documentclass{birkjour}
\usepackage{amsmath}
\usepackage{txfonts}
\usepackage{stmaryrd}
\usepackage{amssymb}
\usepackage{amsfonts}
\usepackage{times}

 \newtheorem{thm}{Theorem}[section]
 
 \newtheorem{lem}[thm]{Lemma}
 
 \theoremstyle{definition}
 
 \theoremstyle{remark}

 \numberwithin{equation}{section}

\newcommand{\D}{{\mathbb D}}
\newcommand{\C}{{\mathbb C}}

\def\D(a,r){\mathbb D(a,r)}
\def\D\m\D(a,r_0){\Delta\setminus\Delta(a,r_0)}

\def\D{\mathbb D}

\begin{document}

%-------------------------------------------------------------------------
% editorial commands: to be inserted by the editorial office
%
%\firstpage{1} \volume{228} \Copyrightyear{2004} \DOI{003-0001}
%\seriesextraline{This is the Concrete Title of this Book\br H.E. R and S.T.C. W, Eds.}
%\firstpage{1}

\title[Characterizations of Dirichlet-type Spaces]
 {Characterizations of Dirichlet-type Spaces}

%----------Author 1
\author[Xiaosong Liu]{Xiaosong Liu}

\address{%
    Department of Mathematics\\
    Shantou University\\
    Shantou 515063, China}

\email{gdxsliu@163.com}

\thanks{This work was  supported by NNSF of China (Grant No. 11171203, 11201280), New Teacher's Fund for Doctor Stations, Ministry of Education  (Grant No.20114402120003) and NSF of Guangdong
Province (Grant No. 10151503101000025, S2011010004511, S2011040004131). }
%----------Author 2
\author{Gerardo R. Chac\'on$^\dag$}\thanks{$^\dag$ This author is partially supported by Pontificia Universidad Javeriana, (Research Proyect No. 5533)}
\address{Departamento de Matematicas\\
   Pontificia Universidad Javeriana\\
   Cra. 7 No. 43-82\\
    Bogotá, Colombia}
\email{chacong@javeriana.edu.co}
%-------------------Author 3
\author{Zengjian Lou*}\thanks{$^*$Corresponding author.}
\address{Department of Mathematics\\
    Shantou University\\
    Shantou 515063, China}
\email{zjlou@stu.edu.cn}



\date{\today}

\begin{abstract}
We give three characterizations of the Dirichlet-type spaces $D(\mu)$. First we characterize $D(\mu)$ in terms of a double integral and in terms of the mean oscillation in the Bergman metric, none of them involve the use of derivatives. Next, we obtain another characterization for $D(\mu)$ in terms of higher order derivatives.  Also, a decomposition theorem for $D(\mu)$ is established.
\end{abstract}

%%% ----------------------------------------------------------------------
\maketitle
%%% ----------------------------------------------------------------------
%\tableofcontents
\section{Introduction}
Let $\D$ be the unit disk and $H(\D)$ be the analytic function on $\D$. Given a positive Borel measure $\mu$ defined on the boundary of the unit disc $\partial\D$ denote by $P_\mu$ the positive harmonic function defined on the unit disc $\D$ as \[P_\mu(z)=\displaystyle\int_0^{2\pi}\frac{1-|z|^2}{|e^{it}-z|^2}\frac{d\mu(t)}{2\pi}.
\]

The Dirichlet type space $D(\mu)$ is defined as the space of all analytic functions on $\D$ such that \[\int_{\D}|f'(z)|^2 P_\mu(z) dA(z)<\infty.\]
It was shown in \cite{Rich91} that the space $D(\mu)$ is contained as a set in the Hardy space $H^2$, consequently a norm on $D(\mu)$ can be defined as \[\|f\|_{D(\mu)}^2:=\|f\|_{H^2}^2+\int_{\D}|f'(z)|^2 P_\mu(z) dA(z).\]
If $\mu=0$, then define $D(\mu)=H^2$. Notice that if $d\mu=dm$ is the arc-length Lebesgue measure on $\partial\D$, then the Dirichlet-type space $D(m)$ coincides with the classical Dirichlet space $\mathcal{D}$.

Dirichlet-type spaces were introduced by Richter in \cite{Rich91} when investigating analytic two-isometries. These spaces have been studied ever since by several authors, see for example \cite{AA}, \cite{GR1}, \cite{GR2}, \cite{CFS}, \cite{CR1}, \cite{CR2}, \cite{Rich91}, \cite{SD}, \cite{Sa2} and \cite{SS}.

The aim of this article is to give characterizations of the spaces $D(\mu)$. We give a characterization of the spaces $D(\mu)$ which avoids the use of derivatives and a characterization in terms of  the mean oscillation  in the Bergman metric.  We also  give a characterization that makes use of high-order derivatives. Finally, as the main result of this paper, we establish an atomic decomposition theorem for  $D(\mu)$.

Derivative-free and higher-order derivatives characterizations of function spaces have received attention in the last years. In \cite{BP} this problem is studied in the setting of Besov spaces in order to characterize the boundedness of certain type of Hankel operators. In \cite{WZ1} and \cite{WZ} the problem is studied for $Q_p$ spaces. The problem of finding an atomic decomposition for a given function space has been extensively studied. For example, it has been established in the case of Bloch spaces, Dirichlet space, BMOA, VMOA and $Q_p$ spaces. We refer to \cite{R}, \cite{RS2}, \cite{RW}, \cite{WX} and the references therein.

The article is distributed as follows. In the following section we give some preliminary notions. In Section 3, we show two derivative-free characterizations of $D(\mu)$ whereas in Section 4 we give a further characterization based on higher-order derivatives. The  decomposition theorem of  $D(\mu)$ is shown in Section 5.

\section{Notations}
For a positive finite Borel measure $\mu$ on  $\partial\D$, we consider the family of functions
\[
P_{\mu_{r}}(z)=\int_{\partial\D}\frac{r^{2}(1-|z|^2)}{|\zeta-rz|^2}d\mu(\zeta),\,\,\,\,z\in\D,\,r\in(0,1).\eqno(2.1)
\]
It is well-known (\cite{SS}) that $P_{\mu_{r}}(z)$ is a subharmonic function and
\[
\lim_{r\rightarrow 1^{-}}P_{\mu_{r}}(z)=P_{\mu}(z).\eqno(2.2)
\]

Following \cite{RS}, we define the local Dirichlet integral of $f$ at $\lambda\in\partial\D$ as
\[
D_{\lambda}(f)=\frac{1}{2\pi}\int_{0}^{2\pi}\Big|\frac{f(e^{it})-f(\lambda)}{e^{it}-\lambda}\Big|^{2}dt.
\]
If $\mu$ is a positive finite Borel measure on  $\partial\D$, we have a representation of the norm of $f\in D(\mu)$ as a consequence of the following formula showed in \cite[Proposition 2.2]{RS}
\[
\int_{\partial\D}D_{\zeta}(f)d\mu(\zeta)=\int_{\D}|f'(z)|^{2}P_{\mu}(z)dA(z).
\]

Give a finite and positive Borel measure $\nu$ on $\D$, we  say that $\nu$ is a $\mu$-Carleson measure if there exists a constant $C$  independent of $f$ such that for all $f\in D(\mu)$ (\cite{GR1} and \cite{CFS})
\[
\int_{\D}|f(z)|^{2}d\nu(z)\leq C\|f\|_{D(\mu)}^{2}.
\]

Throughout the article, we will denote by $C$ a  positive constant that may differ from line to line. The notation  $F\lesssim G$ means that  there exists a constant $C>0$ such that $F\leq CG$ and $C$ is independent of the functions and variables in the inequality. The notation $F\approx G$  indicates that $F\lesssim G$ and also $G\lesssim F$.

\section{A double integral characterization of $D(\mu)$ spaces}
In this section, we characterize  the Dirichlet-type spaces $D(\mu)$ in terms of a double integral. Also, we give a characterization of $D(\mu)$ in terms of the mean oscillation in the Bergman metric. Similar characterizations in other spaces have been studied in \cite{BP}, \cite{WZ1} and \cite{WZ}.

Let $d(z,w)$ denote the Bergman metric between two points in $z,w\in\D$:
\[
d(z, w)=\log\frac{1+|\varphi_{z}(w)|}{1-|\varphi_{z}(w)|},\,\,\,\,\,\,\,z,\,w\in\D,
\]
where $\varphi_{z}(w)=\displaystyle\frac{z-w}{1-\bar{z}w}$. For $z\in\D$ and $R>0$, we denote by $B(z,R)=\{w\in\D: d(z,w)<R\}$   the Bergman ball at $z$ with radius $R$  and by $|B(z,R)|$ the area of $B(z,R)$. If $R>0$ is fixed, then it is well-known that $|B(z,R)|$ is comparable to $(1-|z|^{2})^{2}$ as $|z|\to 1^-$ (see, for example, Section 4.2 of \cite{Zhu1}). Given a function $f\in L^{2}(\D,dA)$, we define  the mean oscillation of $f$ as
\[
MOf(z)=\left(\int_{\D}|f\circ\varphi_{z}(w)-f(z)|^{2}dA(w)\right)^{\frac{1}{2}}.
\]
For $0<r<1$ fixed, let
\[
\hat{f_{r}}(z)=\frac{1}{|B(z,r)|}\int_{B(z,r)}f(w)dA(w)
\]
denote the average of $f$ over the  Bergman ball $B(z,r)$. The mean oscillation of $f$ at $z$ in the
Bergman metric is defined by
\[
MO_{r}f(z)=\left(\frac{1}{|B(z,r)|}\int_{B(z,r)}|f(w)-\hat{f_{r}}(z)|^{2}dA(w)\right)^{\frac{1}{2}}.
\]

\noindent It is easy to check that for $z\in\D$ we have
\begin{eqnarray*}
(MO_{r}f(z))^{2}&=&\widehat{|f|_{r}^{2}}(z)-|\hat{f_{r}}(z)|^2\nonumber\\
&=&\frac{1}{|B(z,r)|^{2}}\int_{B(z,r)}\int_{B(z,r)}|f(u)-f(v)|^{2}dA(u)dA(v).
\end{eqnarray*}

In order to show our first result, two  lemmas are needed. A proof of the following lemma can be found in \cite[Lemma 3.5]{AN} (see also \cite[Lemma 1]{ZR}).
\begin{lem}
Suppose that $\eta,\, \zeta,\,z \in\D$. Let $s > -1,\, r,\, t > 0$ and   $t < s + 2 < r$. Then
\[
\int_{\D}\frac{(1-|\eta|^{2})^{s}}{|1-\bar{\eta}z|^{r}|1-\bar{\eta}\zeta|^{t}}dA(\eta)\leq \frac{C}{(1-|z|^{2})^{r-s-2}|1-\bar{\zeta}z|^{t}}.
\]
\end{lem}

\begin{lem}\label{lemestPoiss}
Let $s>-2$ and  $p>s+3$. Then
\[
\int_{\D}\frac{(1-|w|^2)^{p-s-2}(1-|z|^2)^{s}}{|1-\bar{w}z|^{p}}P_{\mu}(z)dA(z)\leq C P_{\mu}(w).
\]
\end{lem}

\begin{proof}

We first  show that
\[
\int_{\D}\frac{(1-|w|^2)^{p-s-2}(1-|z|^2)^{s}}{|1-\bar{w}z|^{p}}P_{\mu_{r}}(z)dA(z)\leq C P_{\mu_{r}}(w).
\]
From Lemma 3.1 and  Fubini's theorem, we have
\begin{eqnarray*}
&&\int_{\D}\frac{(1-|w|^2)^{p-s-2}(1-|z|^2)^{s}}{|1-\bar{w}z|^{p}}P_{\mu_{r}}(z)dA(z)\nonumber\\
&=&\int_{\D}\frac{(1-|w|^2)^{p-s-2}(1-|z|^2)^{s}}{|1-\bar{w}z|^{p}}\int_{\partial\D}
\frac{r^{2}(1-|z|^2)}{|\zeta-rz|^2}d\mu(\zeta)dA(z)\nonumber\\
&\leq&C\int_{\partial\D}\frac{r^{2}(1-|w|^2)}{|\zeta-rw|^2}d\mu(\zeta)=C P_{\mu_{r}}(W).
\end{eqnarray*}
Letting $r\rightarrow 1^{-}$ and using Fatou's Lemma, we get
\begin{eqnarray*}
&&\int_{\D}\frac{(1-|w|^2)^{p-s-2}(1-|z|^2)^{s}}{|1-\bar{w}z|^{p}}P_{\mu}(z)dA(z)\nonumber\\
&=&\int_{\D}\frac{(1-|w|^2)^{p-s-2}(1-|z|^2)^{s}}{|1-\bar{w}z|^{p}}{\underline{\lim}}_{r\rightarrow 1^{-}}P_{\mu_{r}}(z)dA(z)\nonumber\\
&\leq&\underline{\lim}_{r\rightarrow 1^{-}}\int_{\D}\frac{(1-|w|^2)^{p-s-2}(1-|z|^2)^{s}}{|1-\bar{w}z|^{p}}P_{\mu_{r}}(z)dA(z)\nonumber\\
&\leq&C{\underline{\lim}}_{r\rightarrow 1^{-}}P_{\mu_{r}}(w)=CP_{\mu}(w).
\end{eqnarray*}
This finishes the proof.
\end{proof}

We are ready to establish one of the main theorems of this section.

\begin{thm}\label{thmchar1}
Suppose $ \sigma,\, \tau > -1$.  Then $f\in D(\mu)$ if and only if
\[
\int_{\D}\int_{\D}\frac{|f(z)-f(w)|^2}{|1-\bar{z}w|^{4+\sigma+\tau}}
P_{\mu}(z)dA_{\sigma}(z)dA_{\tau}(w)<\infty.
\]
\end{thm}
\begin{proof}
Suppose first that $\sigma \neq \tau$. We may assume that $\sigma > \tau$. Then if
$z,\,w\in\D$, we have (\cite[p.109]{RW})
\begin{eqnarray*}
\frac{(1-|w|^2)^{\sigma}(1-|z|^2)^{\sigma}}{|1-\bar{z}w|^{4+2\sigma}}&\leq& \frac{(1-|w|^2)^{\sigma}(1-|z|^2)^{\tau}}{|1-\bar{z}w|^{4+\sigma+\tau}}\nonumber\\
&\leq& \frac{(1-|w|^2)^{\tau}(1-|z|^2)^{\tau}}{|1-\bar{z}w|^{2\tau+4}}.
\end{eqnarray*}
Consequently, the case $\sigma \neq \tau$ can be obtained from  the case $\sigma = \tau$.

In what follows, we may assume that $\sigma=\tau$.
It is well-known (\cite[Theorem 4.27]{Zhu2}) that for any $F\in H(\D)$
\begin{eqnarray}
\int_{\D}|F(w)-F(0)|^2dA_{\sigma}(w)\approx\int_{\D}|F'(w)|^2(1-|w|^2)^2A_{\sigma}(w).
\end{eqnarray}
From Lemma 3.2 and equation (3.1), we get
\begin{eqnarray}
I(f)&=&\int_{\D}\int_{\D}\frac{|f(z)-f(w)|^2}{|1-\bar{z}w|^{4+2\sigma}}P_{\mu}(z)dA_{\sigma}(z)
dA_{\sigma}(w)\nonumber\\
&=&\int_{\D}\int_{\D}|f(\varphi_{z}(w))-f(\varphi_{z}(0))|^2dA_{\sigma}(w)
\frac{P_{\mu}(z)}{(1-|z|^2)^{2+\sigma}}dA(z)\nonumber\\
&\approx&\int_{\D}\int_{\D}|\big(f(\varphi_{z}(w))\big)'|^2(1-|w|^2)^{2}dA_{\sigma}(w)
\frac{P_{\mu}(z)}{(1-|z|^2)^{2+\sigma}}dA_{\sigma}(z)\nonumber\\
&\approx&\int_{\D}\int_{\D}|f'(w)|^2
\frac{(1-|w|^2)^{\sigma+2}(1-|z|^2)^{\sigma}}{|1-\bar{z}w|^{4+2\sigma}}dA(w)
P_{\mu}(z)dA(z)\nonumber\\
&\leq&C\int_{\D}|f'(w)|^2P_{\mu}(w)dA(w).
\end{eqnarray}

Conversely, for any $f\in H(\D)$, we may apply the following estimates (cf. \cite[Charpter 4]{Zhu2})
\[
|f'(z)|^{2}\leq \frac{C}{(1-|z|^2)^{2+\sigma}}\int_{B(z,r)}|f'(w)|^2dA_{\sigma}(w).
\]
Since
\[
\frac{(1-|w|^2)^2}{|1-z\bar{w}|^{4+\sigma}}\approx\frac{1}{(1-|z|^2)^{2+\sigma}},\,\,\,\,w\in B(z,r).
\]
Using the estimate of $I(f)$ in (3.2) yields
\begin{eqnarray*}
I(f)&\geq&\int_{\D}\int_{B(z,r)}\Big|f'(w)\frac{1-|w|^2}{|1-\bar{z}w|^{2+\sigma}}\Big|^2dA_{\sigma}(w)
P_{\mu}(z)dA_{\sigma}(z)\nonumber\\
&\approx&\int_{\D}\frac{1}{(1-|z|^2)^{2+\sigma}}\int_{B(z,r)}|f'(w)|^2dA_{\sigma}(w)
P_{\mu}(z)dA(z)\nonumber\\
&\geq&\int_{\D}|f'(z)|^2P_{\mu}(z)dA(z).
\end{eqnarray*}
\end{proof}
Now, we give a characterization of $D(\mu)$ in terms of the mean oscillation in the Bergman metric.

\begin{thm}
Let $f\in A^{2}$ and $d\tau(z)=dA(z)/(1-|z|^{2})^2$ on $\D$. Then the following statements are equivalent:
\begin{itemize}
\item[(i)]$f\in D(\mu)$;
\item[(ii)]
$$
\int_{\D}\big(MOf(z)\big)^{2}P_{\mu}(z)d\tau(z)<\infty;
$$
\item[(iii)]
$$
\int_{\D}\big(MO_{r}f(z)\big)^{2}P_{\mu}(z)d\tau(z)<\infty,
$$
where $r$ is any fixed positive radius.
\end{itemize}
\end{thm}
\begin{proof}$(i)\Rightarrow(ii)$. For $f\in A^{2}$, from  \cite[Section 7.1]{Zhu1}, we have
\begin{align}
2\pi\big(MOf(z)\big)^{2}=&\int_{\D}
|f(w)-f(z)|^{2}\frac{(1-|z|^{2})^2}{|1-\bar{z}w|^{4}}dA(w).\nonumber
\end{align}
Thus,
\[
\int_{\D}\int_{\D}\frac{|f(z)-f(w)|^2}{|1-\bar{z}w|^{4}}P_{\mu}(z)dA(z)dA(w)\approx
\int_{\D}\big(MOf(z)\big)^{2}P_{\mu}(z)d\tau(z).
\] and the proof follows form Theorem \ref{thmchar1}.

$(ii)\Rightarrow(iii)$. The proof follows from the fact that (\cite[Theorem 7.1.6]{Zhu1})
\[
MO_{r}f(z)\leq MOf(z).
\]

$(iii)\Rightarrow(i)$. Since (see \cite[p.35]{Xiao} or \cite[p.292]{WZ1})
\[
(1-|z|^{2})|f'(z)|\leq MO_{r}f(z),
\]
then
\begin{eqnarray*}
\int_{\D}|f'(z)|^2P_{\mu}(z)dA(z)
&=&\int_{\D}(1-|z|^{2})^{2}|f'(z)|^{2}P_{\mu}(z)d\tau(z)\nonumber\\
&\leq&\int_{\D}\big(MO_{r}f(z)\big)^{2}d\tau(z).
\end{eqnarray*}
This finishes the proof.
\end{proof}

\section{Higher order derivatives characterization of $D(\mu)$ spaces}
In this section, we show a further characterization of $D(\mu)$ spaces. This time in terms of higher order derivatives. For this, we will need to show the boundedness of certain integral operator by making use of Schur's test. We include it here for the sake of completeness.

Let $(X,\mu)$ be a measure space. For $f\in L^{p}(d\mu)$, we define the integral operator
\[
Tf(x)=\int_{X}H(x,y)f(y)d\mu(y),
\]
where $H$ is a nonnegative and measurable function on $X\times X$.

\begin{lem}(\cite[Corollary 3.2.3]{Zhu1}\label{lemSchur2})
Assume $\mu$ is a $\sigma$-finte measure. If there  exists a positive and measurable function $h$
on $X$ and a positive constant $C>0$ such that
\[
\int_{X}H(x,y)h(y)d\mu(y)\leq Ch(x)
\]
for almost all $x\in X$ and
\[
\int_{X}H(x,y)h(x)d\mu(x)\leq Ch(y)
\]
for almost all $y\in X$, then the integral operator $T$ is bounded on $L^2(X, d\mu)$. Furthermore, the norm of $T$ on $L^2(X, d\mu)$ does not exceed the constant $C$.
\end{lem}

\begin{lem}(\cite[Lemma 4.2.2]{Zhu1}\label{lemest})
Suppose $t>-1$. If $s>0$, then there exists a constant $C$ such that
\[
\int_{\D}\frac{(1-|w|^{2})^{t}}{|1-z\bar{w}|^{2+s+t}}dA(w)\leq \frac{C}{(1-|z|^{2})^{s}}
\]
for all $z\in\D$. If $s<0$, then there exists a constant $C$ such that
\[
\int_{\D}\frac{(1-|w|^{2})^{t}}{|1-z\bar{w}|^{2+s+t}}dA(w)\leq C
\]
for all $z\in\D$.
\end{lem}
Given $\mu$ a finite Borel positive measure on $\partial\D$, define the measure $\nu$ on $\D$ as
\[
d\nu(z)=P_{\mu}(z)dA(z)
\]
and the integral operator
\begin{equation}\label{eqopT}
Tf(z)=\int_{\D}H(z,w)f(w)d\nu(z),\,\,\,\,\,\,f\in L^{2}(d\nu),
\end{equation}
where
\[
H(z,w)=\frac{(1-|z|^2)^n(1-|w|^2)^{\alpha}}{|1-z\bar{w}|^{2+n+\alpha}P_{\mu}(w)}
\]
is a positive integral kernel and $\alpha$ is a sufficiently large constant.

Also,  consider  integral operator $S$ defined as
\begin{equation}\label{eqopS}
Sf(z)=\int_{\D}L(z,w)f(w)d\nu(w),\,\,\,\,\,\,f\in L^{2}(d\nu),
\end{equation}
where
\[
L(z,w)=\frac{(1-|w|^2)^{\alpha}}{|1-z\bar{w}|^{2+\alpha}P_{\mu}(w)}.\eqno(4.3)
\]
Again, $\alpha$ is a sufficiently large constant.

Our goal in this section is to show that the operators $T$ and $S$ are bounded on $L^2(\D,d\nu)$. As a consequence, we will give the announced characterization of $D(\mu)$ in terms of higher order derivatives.

\begin{thm}\label{thmopT}
The operator $T$ defined in \eqref{eqopT} is bounded on $L^2(\D,d\nu)$ for $\alpha$ sufficiently large.
\end{thm}

\begin{proof}
Fix  constants $\sigma$ and $\alpha$ such that
\[
\sigma<n,\quad\alpha>\sigma+1,\quad\alpha+\sigma>-1.
\]
We will apply Lemma 4.1 for the test function
\[
h(z)=(1-|z|^2)^{\sigma},\qquad z\in\D.
\]
Since $\alpha+\sigma>-1$ and $n-\sigma>0$, we may apply Lemma \ref{lemest} to conclude that there exists a constant $C>0$, such that
\begin{eqnarray*}
\int_{\D}H(z,w)h(w)d\nu(w)&=&(1-|z|^2)^{n}\int_{\D} \frac{(1-|w|^2)^{\alpha+\sigma}}{|1-z\bar{w}|^{2+(\alpha+\sigma)+(n-\sigma)}}dA(w)\nonumber\\
&\leq& Ch(z)
\end{eqnarray*}
for all $z\in\D$.

Next, for any $w\in\D$, applying Lemma \ref{lemestPoiss}, we get
\begin{eqnarray*}
\int_{\D}H(z,w)h(z)d\nu(z)&=&\frac{(1-|w|^2)^{\sigma}}{P_{\mu}(w)}\int_{\D}
\frac{(1-|w|^2)^{\alpha-\sigma}(1-|z|^2)^{n+\sigma}P_{\mu}(z)}{|1-z\bar{w}|^{2+n+\alpha}}dA(z)\nonumber\\
&\leq&Ch(w).
\end{eqnarray*}
And as a consequence of Lemma \ref{lemSchur2}, the proof of the theorem is now complete.
\end{proof}

\begin{thm}\label{thmopS}
 The operator $S$ defined in \eqref{eqopS} is bounded on $L^2(\D,d\nu)$ for $\alpha$ sufficiently large.
\end{thm}
\begin{proof}
For $0<\epsilon<1$ and $\alpha>-\epsilon+1$, we consider the function
\[
h(z)=(1-|z|^2)^{-\epsilon},\qquad z\in\D.
\]
Again, we will apply Lemma \ref{lemSchur2} to show the boundedness of $S$ on $L^2(\D,d\nu)$.

First, for any $z\in\D$ , from Lemma \ref{lemest}, we have
\begin{eqnarray*}
\int_{\D}L(z,w)h(w)d\nu(w)&=&\int_{\D} \frac{(1-|w|^2)^{\alpha-\epsilon}}{|1-z\bar{w}|^{2+(\alpha-\epsilon)+\epsilon}}dA(w)\nonumber\\
&\leq& Ch(z).\
\end{eqnarray*}

Next, for any $w\in\D$, using Lemma \ref{lemestPoiss} again, we obtain
\begin{eqnarray*}
\int_{\D}L(z,w)h(z)d\nu(z)&=&\frac{(1-|w|^2)^{-\epsilon}}{P_{\mu}(w)}\int_{\D}
\frac{(1-|w|^2)^{\alpha+\epsilon}(1-|z|^2)^{-\epsilon}P_{\mu}(z)}{|1-z\bar{w}|^{2+\alpha}}dA(z)\nonumber\\
&\leq&Ch(w).
\end{eqnarray*}
Hence, the boundedness of $S$ on $L^2(\D,d\nu)$ follows.
\end{proof}

\begin{thm}
Let $n$ be any nonnegative integer. Then $f\in D(\mu)$ if and only if
\[
\int_{\D}|f^{(n+1)}(z)|^{2}(1-|z|^2)^{2n}P_{\mu}(z)dA(z)<\infty.\eqno(4.4)
\]
\end{thm}

\begin{proof}
Suppose that $f\in D(\mu)$, then  it has the following integral representation:
\[
f'(z)=(\alpha+1)\int_{\D}\frac{f'(w)(1-|w|^2)^{\alpha}}{(1-z\bar{w})^{2+\alpha}}dA(w),\,\,\,\,z\in\D,\,\alpha>1.
\]
Differentiating under the integral sign $n$ times and multiplying the result by $(1-|z|^2)^{n}$, we have
\[
(1-|z|^2)^{n}f^{(n+1)}(z)=C\int_{\D}\frac{(1-|z|^2)^{n}
(1-|w|^2)^{\alpha}\bar{w}^{n}f'(w)}{(1-z\bar{w})^{2+\alpha+n}}dA(w),
\]
where $C$ is a positive constant depending only on $\alpha$ and $n$. In particular,
\[
(1-|z|^2)^{n}\big|f^{(n+1)}(z)\big|\leq C\int_{\D}H(z,w)|f'(w)|d\nu(w).
\]
From Theorem \ref{thmopT}, we obtain
\[
\int_{\D}(1-|z|^2)^{2n}|f^{(n+1)}(z)|^2d\nu(z)\leq C\int_{\D}|f'(z)|^2d\nu(z).
\]

Conversely, integrating $n$-times both sides of the following representation
 (see, for example, \cite[Corollary 1.5]{HKZ} or \cite[Corollary 8]{WZ} ),
\[
f^{(n+1)}(z)=(n+\alpha+1)\int_{\D}\frac{f^{(n+1)}(w)(1-|w|^{2})^{n}(1-|w|^{2})^{\alpha}dA(w)}{(1-z\bar{w})^{2+n+\alpha}},
\]
we get
\[
f'(z)=\int_{\D}\frac{h(z,w)f^{(n+1)}(w)(1-|w|^{2})^{n}(1-|w|^{2})^{\alpha}dA(w)}{(1-z\bar{w})^{2+\alpha}},
\]
where $h(z,w)$ is a bounded function in $z$ and $w$. In particular, we have
\[
|f'(z)|\leq C\int_{\D}L(z,w)|f^{(n+1)}(w)|(1-|w|^{2})^{n}d\nu(w)
\]
and from Theorem \ref{thmopS}, we obtain
\[
\int_{\D}|f'(z)|^{2}d\nu(z)\leq C\int_{\D}\big|f^{(n+1)}(z)\big|^{2}(1-|z|^{2})^{2n}d\nu(z)\]
and the result follows.
\end{proof}

\section{Decomposition theorem for $D(\mu)$ spaces}
In this section, as a main result of the article, we show a decomposition theorem for Dirichlet-type spaces $D(\mu)$. Decomposition theorems in different function spaces such as Bergman spaces, Bloch spaces, Dirichlet spaces, BMOA space and $Q_p$ spaces have been established and proved its usefulness in several articles. See, for example,  \cite{R}, \cite{RS}, \cite{RW}
and \cite{WX}.

We will say that a sequence of points $\{z_{j}\}_{j=1}^{\infty} \in \D$ is {\em $\eta$-separated}, if there exists $\eta>0$ such that
\[
\inf_{j\neq k}d(z_{j},z_{k})\geq \eta.
\]
On the other hand, we will say that $\{z_{j}\}_{j=1}^{\infty}$ is {\em $\eta$-dense} if
\[
\D=\bigcup_{j=1}^{\infty}B(z_{j},\eta).
\]

The following two lemmas are standard and their proofs can be found in \cite[Lemmas 4.10 and 4.7]{Zhu2}.

\begin{lem}\label{lemdecomp1}
For any $\eta\in (0,1)$, there is a $\frac{\eta}{2}$-separated and $\eta$-dense sequence $\{z_j\}_{j=1}^{\infty}\subset\D$ and Lebesgue measurable sets $D_{j}$ (j=1,2,$\cdots$) such that:
\begin{itemize}
\item[(1)]$B(z_{j},\frac{\eta}{4})\subset D_{j} \subset B(z_{j},\eta)$;
\item[(2)]$D_{i}\cap D_{j}=\varnothing$, if $i\neq j$;
\item[(3)]$\D=\bigcup^{\infty}_{j=1}D_{j}$.
\end{itemize}
\end{lem}
\begin{lem}\label{lemdecomp2}
For any $\eta\in (0,1)$ and $N\in \mathbb{N}$, there is an $\frac{\eta}{2}$-separated and $\eta$-dense sequence $\{z_{j}\}_{j=1}^{\infty}\subset\D$ such that any $z\in \D$ lies in at most $N$ of the sets $B(z_{j},2\eta)$ (j=1,2,$\cdots$).
\end{lem}

We will also need the following three lemmas which can be found in \cite{R} or \cite{RS}.
\begin{lem}\label{lemkernel}
 If $z_{0}\in\D$ and $\eta\leq 1$,
there exists a constant $C>0$, independent of $\eta$ and $z_0$, such that
\[
|k_w(z)-k_w(z_0)|\leq C\eta|k_w(z)|,
\]
for all $w\in\D$ and $z\in B(z_0,\eta)$, where
\[
k_w(z)=\frac{(1-|z|^2)^{b-1}}{(1-\bar{z}w)^{b+1}},\qquad
b>0.
\]
\end{lem}

\begin{lem}
Let $0<\eta<\frac{1}{4}$ and $\{z_{j}\}_{j=0}^{\infty}$ be an $\eta$-separated. There exists a constant $C>0$ such that for any  $f\in H(\D)$ and for all $j=1,2,\cdots$.

\[
\int_{D_{j}}|f(z)-f(z_{j})|dA(z)\leq C\eta^{3}\int_{B({z_{j},\frac{\eta}{4}})}|f(z)|dA(z)
\]
\end{lem}

\begin{lem}
Let $0<\eta<1$ and $\{z_{j}\}_{j=1}^{\infty}$ be an $\eta$-separated. There exists a positive integer $N_{0}=O\left(\eta^{-2}\right)$ such that each point of  $\D$ lies in at most $N_{0}$ of the discs in $\{B({z_{j},\frac{\eta}{4}})\}_{j=1}^{\infty}$. Furthermore, if $b>0$ and $f$ is analytic on $\D$, then
\[
\sum_{j=1}^{\infty}\int_{B({z_{j},\frac{\eta}{4}})}|f(z)|^{2}(1-|z|^{2})^{b-1}dA(z)
\leq N_{0}\int_{\D}|f(z)|^{2}(1-|z|^{2})^{b-1}dA(z).
\]
\end{lem}

\begin{lem}\label{lemPmu}
Let $\mu$ be a positive finite Borel measure on $\partial\D$. For $\eta\in (0,1)$, let $\{z_j\}^\infty_{j=1}\subset\D$ be an $\eta$-separated sequence. If $z\in B(z_j, \eta),\,j=1,2,\cdots,$ then there exist two positive constants $C_1$ and $C_2$ such that
\[
C_1P_{\mu}(z_{j})\leq P_{\mu}(z)\leq
C_2P_{\mu}(z_{j}),\,\,\,\,j=1,2,\cdots.
\]
\end{lem}
\begin{proof} Let $z\in B(z_j, \eta), j=1,2\cdots$ and $r\in(0,1)$. It is easy to check that there exists a constant $C>0$, independent of the sequence
$\{z_j\}^\infty_{j=1}$ and $\eta$ such that
\[
|1-r\bar{\zeta}z| \leq C|1-r\bar{\zeta}z_j|,\qquad \zeta\in \partial\D.
\]
Also, by  Lemma 4.3.4 in \cite{Zhu1}, there is a constant $C>0$ independent of $\{z_j\}^\infty_{j=1}$ and $\eta$ such that, for $z\in B(z_j, \eta)$
\[
1-|z_j|^2\leq C(1-|z|^2).
\]
Therefore,
\[
\frac{r^{2}(1-|z_j|^2)}{|1-r\bar{\zeta}z_j|^2}\leq C\frac{r^{2}(1-|z|^2)}{|1-r\bar{\zeta}z|^2}.
\]
Integrating on $\partial\D$ with respect to $\mu$ and letting $r\to 1^{-}$, we have
\[
P_{\mu}(z_{j})\leq CP_{\mu}(z).
\] The other inequality follows in a similar way.
\end{proof}

\begin{lem}\label{lemmPoiss2}
Let $f\in D(\mu)$ and $\{z_j\}^{\infty}_{j=1}$ be an $\eta$-separated, $0<\eta<1$. Then
\[
\sum^{\infty}_{j=1}(1-|z_j|)^2|f'(z_{j})|^2P_{\mu}(z_j)\leq C\int_{\D}|f'(z)|^2P_{\mu}(z)dA(z)
\]
\end{lem}
\begin{proof}
For any  $f\in H(\D)$, we have (see \cite[Proposition 4.13]{Zhu2})
\[
|f'(z_j)|^2\leq \frac{C}{|B(z_j, \eta)|}\int_{B(z_j, \eta)}|f'(w)|^2dA(w).
\]
Note that $|B(z_{j},\eta)|=(1-|z_j|^2)^2$, from Lemma 5.6, we obtain
\begin{eqnarray*}
\sum^{\infty}_{j=1}(1-|z_j|)^2|f'(z_{j})|^2P_{\mu}(z_j)
&\leq & C\sum^{\infty}_{j=1}\int_{B(z_j, \eta)}|f'(w)|^2P_{\mu}(w)dA(w)\\
&\leq & C\int_{D}|f'(w)|^2P_{\mu}(w)dA(w)
\end{eqnarray*}
\end{proof}

Now we are ready to prove the main theorem.

\begin{thm}[Decomposition Theorem]
Let $\mu$ be a nonnegative Borel measure on $\partial\D$ and $b>2$. Then, there exists $d_0>0$ such that for any $d$-separated $\{z_j\}_{j=1}^\infty$ in $\D$ ($0<d<d_0$), we have,
\begin{itemize}
\item[(i)] If $f\in D(\mu)$, then there exists a sequence $\{\lambda_j\}_{j=1}^{\infty}\subset \C$ such that
    \begin{equation}\label{eqdecomp}
    f(z)=f(0)+ \sum_{j=1}^\infty \lambda_j(1-|z_j|^2)^b\left(\frac{1}{(1-\overline{z_j}z)^b}-1\right)
    \end{equation}
    and \[\sum_{j=1}^\infty |\lambda_j|^2P_\mu(z_j)\leq C\|f\|^2_{D(\mu)}.\]
\item[(ii)] If a sequence $\{\lambda_j\}_{j=1}^{\infty}\subset \C$ satisfies that $\sum_{j=1}^\infty |\lambda_j|^2 P_\mu (z)\delta_{z_j}$ is a $\mu$-Carleson measure, then the series defined in (5.1) converges in $D(\mu)$ and \[\|f\|^2_{D(\mu)}\leq C\sum_{j=1}^\infty |\lambda_j|^2 P_\mu (z_j).\]
\end{itemize}
\end{thm}

\begin{proof}
\noindent For part (i), recall that an equivalent norm for the Dirichlet type spaces $D(\mu)$ is given by (see, for example, \cite[Lemma 2.3]{GR3})

\[\|f\|^2_{D(\mu)}\thickapprox|f(0)|^2 + \int_{\D}|f'(z)|^2P_{\mu}(z)dA(z).\]

If we define the space $D_0(\mu):=\{f\in D(\mu): f(0)=0\}$ with the norm
\[
\|f\|_{D_0(\mu)}=\Big(\int_{\D}|f'(z)|^{2}P_{\mu}(z)dA(z)\Big)^{\frac{1}{2}},
\] then $f-f(0)\in D_0(\mu)$ for $f\in D(\mu)$. Moreover, the space $D(\mu)$ can be written as \[D(\mu)=D_0(\mu)\oplus \C.\]

For $b>2$, assume that $f\in D_0(\mu)$. Then $f\in H^{2}$ and $f'\in A_{1}^{2}\subset A_{b-1}^{2}$, the weighted Bergman spaces. By the reproducing formula of the Bergman space, we have
\[
f'(z)=\frac{b}{\pi}\int_{\D}\frac{(1-|w|^2)^{b-1}}{(1-\bar{w}z)^{b+1}}f'(w)dA(w).
\]
Since $\{z_j\}^{\infty}_{j=1}$ is $\frac{\eta}{2}$-separated and $\eta$-dense, then there exists a disjoint partition $\{D_j\}^{\infty}_{j=1}$ of $\D$
\[
f'(z)=\frac{b}{\pi}\sum^{\infty}_{j=1}\int_{D_j}\frac{(1-|w|^2)^{b-1}}{(1-\bar{w}z)^{b+1}}f'(w)dA(w).
\]
Now define the linear operator $A$ on $D_0(\mu)$ by
\[
A(f)(z)=\frac{b}{\pi}\sum^{\infty}_{j=1}f'(z_j)|D_j|\frac{(1-|z_j|^2)^{b-1}}{\bar{z_j}(1-\bar{z_j}z)^{b}}.
\]
We will show first that
\begin{equation}\label{eqAinv}
\|f-A(f)\|_{D_0(\mu)}\leq C\eta\|f\|_{D_0(\mu)}.
\end{equation}
Notice that
\begin{eqnarray*}
\big|f'(z)-A(f)'(z)\big|&\leq&\frac{b}{\pi}\sum^{\infty}_{j=1}\int_{D_j}|f'(w)||k_z(w)-k_z(z_j)|dA(w)\nonumber\\
&&+\frac{b}{\pi}\sum^{\infty}_{j=1}\int_{D_j}|f'(w)-f'(z_j)||k_z(z_j)|dA(w)\nonumber\\
&=&I_1+I_2.
\end{eqnarray*}
From Lemma \ref{lemkernel}, we get
\[
I_1\leq C\eta \int_{D}|f'(w)||k_z(w)|dA(w).
\]
From \cite[p.394]{WX}, we have
\[
I_2\leq C\eta\int_{\D}|f'(w)||k_z(w)|dA(w).
\]
Using Theorem \ref{thmopS} yields
\begin{eqnarray*}
\int_{\D}|f'(z)-A(f)'(z)|^{2}P_{\mu}(z)dA(z)
&\leq&C\eta\int_{\D}\Big|\int_{\D}|k_{z}(w)||f'(w)|dA(w)\Big|^{2}P_{\mu}(z)dA(z)\nonumber\\
&\leq&C\eta\int_{\D}|f'(z)|^{2}P_{\mu}(z)dA(z).
\end{eqnarray*}
Thus, inequality (5.2) holds.

Now, define the operator $\mathcal{A}:D(\mu)\to D_0(\mu)$ as
\[
\mathcal{A}(f-f(0))(z):=\frac{1}{\pi}\sum_{j=1}^\infty f'(z_j)|D_j|\frac{(1-|z_j|^2)^{b-1}}{\overline{z_j}} \left(\frac{1}{(1-\overline{z_j}z)^b}-1\right).
\]

In other words, $\mathcal{A}$ is the operator $A$ followed by the projection into the space $D_0(\mu)$.

Consider the operator $B:D(\mu)\to D(\mu)$ defined as \[B=\left(
                                                                  \begin{array}{cc}
                                                                    \mathcal{A} & 0 \\
                                                                    0 & 1 \\
                                                                  \end{array}
                                                                \right)\]
Then, using inequality (5.2), we get
\begin{eqnarray*}
\|(I-B)f\|^2_{D(\mu)} &=& \|f-\mathcal{A}(f-f(0))-f(0)\|^2_{D(\mu)}\\
&=& \|f-A(f-f(0))-A(f-f(0))(0)-f(0)\|^2_{D(\mu)}\\
&=& \|(f-f(0))-A(f-f(0))\|^2_{D_0(\mu)}\\
&\leq& C\eta \|f-f(0)\|_{D_0(\mu)}\\
&\leq& C\eta \|f\|_{D(\mu)},
\end{eqnarray*}
where $I$ stands for the identity operator acting on $D(\mu)$.
Taking $\eta>0$ small enough, we have the invertibility of the operator $B$. Its bounded inverse is defined by
\[
B^{-1}=(I-(I-B))^{-1}=\sum^{\infty}_{n=0}(I-B)^{n}.
\]
We have constructed an approximation operator $B$ with bounded inverse. For any $f\in D(\mu)$, we can write
\begin{eqnarray*}
f(z)&=& BB^{-1}f(z)\\
 &=& \mathcal{A}\mathcal{A}^{-1}(f-f(0))(z)+f(0)\\
 &=& \frac{b}{\pi}\sum_{j=1}^\infty (\mathcal{A}^{-1}(f-f(0)))'(z_j)|D_j|\frac{(1-|z_j|^2)^{b-1}}{\overline{z_j}} \left(\frac{1}{(1-\overline{z_j}z)^b}-1\right)+f(0)\\
 &=& f(0)+ \sum_{j=1}^\infty \lambda_j(1-|z_j|^2)^b\left(\frac{1}{(1-\overline{z_j}z)^b}-1\right),
\end{eqnarray*}
where \[\lambda_j=\frac{b(\mathcal{A}^{-1}(f-f(0)))'(z_j)|D_j|}{\pi \overline{z_j}(1-|z_j|^2)}.\]

We now to show that
\[
\sum^\infty_{j=1}|\lambda_j|^2P_{\mu}(z_j) <C\|f\|_{D(\mu)}^{2}.\eqno(5.3)
\]
By the mean value theorem, we have
\begin{eqnarray*}
\sum^{\infty}_{j=1}|\lambda_j|^2 P_{\mu}(z_{j})
&\leq&C\sum^{\infty}_{j=1}\frac{|D_j|^{2}}{(1-|z_j|^2)^{2}}|\mathcal{A}^{-1}(f-f(0))'(z_j)|^{2}P_{\mu}(z_{j})\nonumber\\
&\leq&C\sum^{\infty}_{j=1}\frac{|D_j|^{2}(1-|z_j|^2)^{-2}}{|B(z_j,\frac{\eta}{4})|}
\int_{B(z_j,\frac{\eta}{4})}|\mathcal{A}^{-1}(f-f(0))'(z)|^{2}P_{\mu}(z)dA(z)\nonumber\\
&\leq&C\int_{\D}|\mathcal{A}^{-1}(f-f(0))'(z)|^{2}P_{\mu}(z)dA(z)\nonumber\\
&\leq&\|\mathcal{A}^{-1}(f-f(0))\|_{D_0(\mu)}^{2}
\leq C\|f\|_{D(\mu)}^{2}.
\end{eqnarray*}
Thus (5.3) is proved.

Next we  show part (ii). Suppose that $\sum_{j=1}^\infty |\lambda_j|^2 P_\mu (z)\delta_{z_j}$ is a $\mu$-Carleson measure. Then there exists a constant $C>0$ such that for every function $f\in D(\mu)$,
\[
\sum_{j=1}^\infty |\lambda_j|^2 |f(z_j)|^2P_\mu(z_j)\leq C\|f\|^2_{D(\mu)}.
\]In particular, if $f\equiv 1$, we have that
\[
\sum_{j=1}^\infty |\lambda_j|^2P_\mu(z_j)<C.\eqno(5.4)
\]

It is sufficient to show that $f\in D(\mu)$ for $f$ defined as in equation \eqref{eqdecomp}. In this case,
\[
f'(w)=b\sum^{\infty}_{j=1}\lambda_j\bar{z_j}\frac{(1-|z_j|^2)^b}{(1-\bar{z_j}w)^{b+1}}.
\]
Applying the substitution $z=\frac{z_j-\zeta}{1-\bar{z_j}\zeta}$ in the following integration, we know that there is a positive constant (cf. \cite{WX})
\[
C_j=\frac{\pi(\frac{e^{\eta}-1}{e^{\eta}+1})^2}{1-(\frac{e^{\eta}-1}{e^{\eta}+1})^2|z_j|^2}
\frac{2b}{1-(\frac{4e^{\eta}}{(e^{\eta}+1)^2})^b},\,\,\,\,\,j=1,2,\cdots,
\]
such that
\[
\int_{B(z_j,\frac{\eta}{4})}\frac{(1-|z|^2)^{b-1}}{(1-\bar{z}w)^{b+1}}dA(z)=
\frac{|B(z_j,\frac{\eta}{4})|}{C_j}\frac{(1-|z_j|^2)^{b-1}}{(1-\bar{z_j}w)^{b+1}}.
\]
Therefore, for $b>2$,
\begin{eqnarray*}
f'(w)&=&b\sum^{\infty}_{j=1}\lambda_j\bar{z_j}C_j\frac{1-|z_j|^2}{|B(z_j,\frac{\eta}{4})|}
\int_{\D}\frac{(1-|z|^2)^{b-1}}{(1-\bar{z}w)^{b+1}}\chi_{B(z_j,\frac{\eta}{4})}(z)dA(z)\nonumber\\
&=&b\int_{\D}\frac{(1-|z|^2)^{b-1}}{(1-\bar{z}w)^{b+1}}
\Big(\sum^{\infty}_{j=1}\lambda_j\bar{z_j}C_j\frac{1-|z_j|^2}{|B(z_j,\frac{\eta}{4})|}
\chi_{B(z_j,\frac{\eta}{4})}(z)\Big)dA(z).
\end{eqnarray*}
Consequently, if
\[
\int_{\D}\Big|\sum^{\infty}_{j=1}\lambda_j\bar{z_j}C_j\frac{1-|z_j|^2}{|B(z_j,\frac{\eta}{4})|}
\chi_{B(z_j,\frac{\eta}{4})}(z)\Big|^2P_{\mu}(z)dA(z)<\infty,\eqno(5.5)
\]
then by Theorem \ref{thmopS}
\[
\int_{\D}|f'(z)|^2P_{\mu}(z)dA(z)\leq C\int_{\D}\Big|\sum^{\infty}_{j=1}\lambda_j\bar{z_j}C_j\frac{1-|z_j|^2}{|B(z_j,\frac{\eta}{4})|}
\chi_{B(z_j,\frac{\eta}{4})}(z)\Big|^2P_{\mu}(z)dA(z).
\]
Hence $f\in D(\mu)$. So, it remains to  show that the inequality (5.5) holds.

Since $\{B(z_j,\frac{\eta}{4})\}^{\infty}_{j=1}$ is a set of disjoint Bergman discs, then the sequence $\{C_j\}^{\infty}_{j=1}$ is bounded. Finally, putting  Lemmas \ref{lemdecomp1}, \ref{lemdecomp2},  \ref{lemPmu} and (5.4), we have
\begin{eqnarray*}
&&\int_{\D}\Big|\sum^{\infty}_{j=1}\lambda_j\bar{z_j}C_j\frac{1-|z_j|^2}{|B(z_j,\frac{\eta}{4})|}
\chi_{B(z_j,\frac{\eta}{4})}(z)\Big|^2P_{\mu}(z)dA(z)\nonumber\\
&\leq&C\int_{\D}\sum^{\infty}_{j=1}|\lambda_j|^2\frac{(1-|z_j|^2)^2}{|B(z_j,\frac{\eta}{4})|^2}P_{\mu}(z)dA(z)\nonumber\\
&\leq&C\sum^{\infty}_{j=1}\int_{B(z_j,\frac{\eta}{4})}|\lambda_j|^2\frac{(1-|z_j|^2)^2}{|B(z_j,\frac{\eta}{4})|^2}P_{\mu}(z)dA(z)\nonumber\\
&\leq&C\sum^{\infty}_{j=1}|\lambda_j|^2P_{\mu}(z_j)\leq C.
\end{eqnarray*}
This finishes the proof.
\end{proof}

\newpage
% ------------------------------------------------------------------------
\end{document}